\documentclass[11pt,a4paper]{amsart}
\usepackage{a4wide}
\usepackage[utf8]{inputenc}
\usepackage{fancyhdr}
\usepackage[all]{xy}
\usepackage[active]{srcltx}
\usepackage{mathtools}
\mathtoolsset{showonlyrefs}
\usepackage{longtable}
\usepackage{afterpage,float}
\usepackage{graphicx}
\usepackage{enumitem}
\usepackage{upgreek}
\usepackage{caption} 
\usepackage{mathrsfs}
\usepackage{hyperref} 
\usepackage{bm}
\usepackage{tikz-cd}
\usepackage{tikz}
\usepackage{tkz-graph}

\newcommand{\blue}[1]{#1}

\theoremstyle{plain}
\newtheorem{theorem}{Theorem} 
\newtheorem{lemma}[theorem]{Lemma}     

\theoremstyle{definition}
\newtheorem{example}[theorem]{Example} 
\newtheorem{definition}[theorem]{Definition} 
\theoremstyle{remark}
\newtheorem{remark}[theorem]{Remark}  

\newcommand{\Hom}{\mathrm{Hom}}

\newcommand\Zn{\mathbb{Z}}

\newcommand\Pn{\mathbb{P}}
\newcommand\bB{\mathbb{B}}
\newcommand\calE{\mathcal{E}}
\newcommand\calF{\mathcal{F}}
\newcommand\calG{\mathcal{G}}

\newcommand\calX{\mathcal{X}}
\newcommand\calO{\mathcal{O}}
\newcommand\calJ{\mathcal{J}}

\newcommand\lhom{\mathcal{H}om}

\newcommand{\tens}[1]{%
  \mathbin{\mathop{\otimes}\displaylimits_{#1}}%
}

\title[Splitting  of supervector bundles]{Splitting   of supervector bundles\\[3pt] on projective superspaces}

\author{Charles Almeida}
\address{ICEx - UFMG \\
Department of Mathematics,   Av. Ant\^onio Carlos, 6627\\
30123-970 Belo Horizonte, MG, Brazil}
\email{charlesalmeida@mat.ufmg.br}

\author{Ugo Bruzzo}
\address{SISSA (Scuola Internazionale Superiore di Studi Avanzati) \\
Via Bonomea 265, 34136 Trieste, Italy \\
INFN (Istituto Nazionale di Fisica Nucleare), Sezione di Trieste \\
IGAP (Institute for Geometry and Physics), Trieste}
\email{bruzzo@sissa.it}

\begin{document}

\maketitle

\begin{abstract}
	We provide a splitting criterion for supervector bundles over the projective superspaces $\mathbb{P}^{n|m}$. More precisely, we prove that a rank $p|q$ supervector bundle on $\mathbb{P}^{n|m}$  with  vanishing intermediate cohomology is  isomorphic to the direct sum of even and odd line bundles, provided that $n \geq 2$. For $n=1$ we provide an example of a supervector bundle that cannot be written as a sum of line bundles. 
 	\medskip
   	 
   	 
   	 \medskip
   	 
\end{abstract}
\section{Introduction}

The mathematical formalization of the study of algebraic supergeometry, extending classical results from algebraic geometry, has opened new directions in exploring mathematical objects arising in physics, particularly in supersymmetric field theories and superstring theory. Supergeometry aims to provide the right environment for studying such structures as moduli space of supervector bundles and moduli space of supersymmetric curves \cite{DW2013,BHR2021}, which are important in the perturbative approach to superstring theory \cite{K2021, M2013}. In this context, understanding the arithmetically Cohen-Macaulay (ACM) bundles is a natural step towards classifying supervector bundles on superschemes. 

ACM bundles are vector bundles on algebraic varieties that have no intermediate cohomology. They are important for the classification of vector bundles, as they provide insights into the structure of the derived category \cite{Kuz2004}. On classical varieties such as projective spaces, Segre varieties,  quadrics, and Grassmannians, ACM bundles were extensively studied, with techniques encompassing Koszul resolutions, derived categories, and the Borel-Weil-Bott theorem (e.g., \cite{AG1998, O89, CMR2005}). However, generalizing these results to the supergeometric setting offers some substantial challenges. Many of the classical tools, such as the Borel-Weil-Bott theorem, Koszul resolutions, or Beilinson-type spectral  either do not exist in the supergeometry context or require substantial modifications, (see  \cite{Noja22, OP84} for the Koszul resolution and \cite{Penkov2022} for the Borel-Weil-Bott theory in the super setting). 

Despite the challenges, extending the notion and the classification of ACM bundles on superschemes is an essential step to progress  toward a more concrete description of supervector bundles. In this paper, we shall focus on the case of ACM bundles on $\Pn^{n|m}$ and shall prove that such bundles split as sums of even and odd line bundles. 

\begin{theorem}
Let $\mathcal{E}$ be a rank $p|q$ supervector bundle on $\Pn^{n|m}$, with $n \geq 2$ and $m \geq 1,$ such that $H^{i}_*(\Pn^{n}, \mathcal{E}) = 0$, for $ 1 \leq i \leq n-1$. Then there exists two sequences of integers $a_1 \geq a_2 \geq \ldots, \geq a_p$, and $b_1 \geq b_2 \geq \ldots, \geq b_q$, such that:

   \[\mathcal{E}  =  \displaystyle \left[ \bigoplus^{p}_i \mathcal{O}_{\Pn^{n|m}}(a_i)\right]\oplus \Pi \left[ \bigoplus^{q}_i\mathcal{O}_{\Pn^{n|m}}(b_i)\right]. \]    
   \end{theorem}
To do so, we shall heavily rely on the split superscheme structure of $\Pn^{n|m}$ (see Definition \ref{splitdef}), and for this reason, the proof differs substantially from the classical ones. Due to their dependency on the split hypothesis, the methods developed here do not directly extend to more general superschemes like the supergrassmannians. 

In Section 2 we shall develop the techniques necessary to relate the properties of sheaves on a superscheme with their direct images on the underlying bosonic reduction $X$. In Section 3 we shall use these tools to prove the splitting criterion.

{\bf Acknowledgements.} The first author would like to thank for the warm hospitality of SISSA, which provided a most welcoming enviroment for his visit during 2024, when this work was carried out. The first author is partially supported by the CNPq grant 408974/2023-0 and by the CAPES/Print Grant nr.~88887.913038/2023-00. The second author's research is partly supported by   PRIN 2022BTA242 ``Geometry of algebraic structures: moduli, invariants, deformations'' and INdAM-GNSAGA.

\bigskip

\section{Preparatory results}
In this section we recall some important aspects of superalgebraic geometry and present some results that will be important for the rest of our work. A reference for the main definitions in this Section is \cite{biblia}. 

A superalgebra $\mathbb A$ is a $\mathbb{Z}_2$-graded agebra which is graded commutative, that is, $a\cdot b = (-1)^{|a|.|b|}b \cdot a$ for every $a,b \in \mathbb A$, where $|c|$ is the $\mathbb{Z}_2-$ degree of the element $c$.

\paragraph{\bf Superschemes.} We say that $\mathbb{B}$ is a supercommutative ring if, in addition to being $\mathbb{Z}_2$-commutative and associative,  the ideal generated by the odd elements $J \subset \bB$ is finitely generated. 
$B=\mathbb{B}/J$ is the bosonic reduction of $\mathbb{B}$. We also say that $\mathbb{B}$ is split if there exists a finitely generated projective $B$-module $M$ such that $\mathbb{B}= {\bigwedge}_AM$. 

We define the odd dimension of $\mathbb{B}$ as
the smallest number of generators of the ideal $J$.

\begin{definition}
A (locally) ringed superspace is a pair $\calX=(X,\mathcal{O}_{\calX})$, where $X$ is a topological space, and $\mathcal{O}_{\calX}$ is a sheaf of $\mathbb{Z}_2$-graded commutative rings such that for every point $x\in X$ the stalk $\mathcal{O}_{\calX,x}$ is a (local)  superring.\footnote{``Local'' means that the superring has a unique maximal homogeneous ideal.}
\end{definition}

 We note that for the category of $\Zn_2$-graded $\calO_\calX$-modules   on a ringed superspace there is an endofunctor $\Pi$ that reverses the parity of the objects, such that $\Pi^2$ is the identity. 

\begin{definition}
    The superspectrum of a superring $\mathbb{A}$ is 
the locally ringed superspace  {$\mathbb{S}pec(\mathbb{A})=(X,\mathcal{O})$}, where $X$ is the spectrum of the bosonic reduction $A$ of $\mathbb{A}$,
and $\mathcal{O}$ is the sheaf of $\mathbb{Z}_2$-graded commutative superrings defined on the fundamental open sets $D(f)$, $f\in \mathbb A$ by $\mathcal O(D_f) = \mathbb A_f$.
A superscheme is a locally ringed superspace $\calX=(X,\mathcal{O}_{\calX})$ which is locally isomorphic to the superspectrum of a superring. 
\end{definition}

 Given a locally ringed superspace $\calX=(X,\mathcal{O}_{\calX})$, we can consider the homogeneous ideal   {$\mathcal{J}$ generated by the odd elements. Then $\mathcal{O}_X:=\mathcal{O}_{\calX}/\mathcal{J}$ is a purely even sheaf of local rings. We shall say that
\begin{equation}\label{superfund}
     0 \to  \mathcal{J} \to  \mathcal{O}_{\calX} \to \mathcal{O}_X \to 0
 \end{equation}
 \noindent is the fundamental exact sequence. We shall call the locally ringed space $X=(X,\mathcal{O}_X)$ the bosonic reduction of $\calX$.

Next we discuss  the notion of dimension for superschemes.

\begin{definition}
     The odd dimension of a superscheme $\calX$ is the supremum of the odd dimensions of the local superrings $ \mathcal{O}_{\calX,x} $ for all points $ x \in X $. The even dimension of $\calX $ is the dimension of the scheme $X$. Both dimensions may be infinite. The dimension of $ X $ is the pair: 
$$
\dim \calX = (\text{even-dim } \calX|\text{odd-dim } \calX).
$$ 
We say that a morphism $ f: \calX \to \mathcal{S} $ of superschemes has relative dimension $ (n|m) $ if the fibres $ X_s $ are superschemes of dimension $ (n|m)$.
 \end{definition}

  Given a $\Zn_2$-graded $\calO_\calX$-module $\mathcal M$, we associate with it the graded module
 $$ \operatorname{Gr} (\mathcal M) = \bigoplus_{i\ge 0} \operatorname{Gr}_i  (\mathcal M)=\bigoplus_{i\ge 0} \mathcal J^i\mathcal M/\mathcal J^{i+1}\mathcal M$$
 of the filtration
 $$ \mathcal M \supset \calJ \mathcal M \supset \calJ^2 \mathcal M \supset ... $$
 Since $\calJ$ annihilates every sheaf  $ \operatorname{Gr}_i  (\mathcal M)$, the latter have a strucure of $\calO_X$-module. 
 
 $ \operatorname{Gr}_0 (\mathcal M) = \mathcal M/\mathcal J\mathcal M$ will be denoted $\mathcal M_{red}$. Of course,
 $(\calO_\calX)_{red}=\calO_X$. One has $\mathcal M_{red}= \mathcal M\otimes_{\calO_\calX}\calO_X$,
and $\mathcal M_{red}$ is an $\calO_X$-module.}
 Note that, if $\calE$ is a rank $p|q$ locally free sheaf on $\calX$, then $\calE_{red}$ can be written as $E_0+\Pi E_1$, where $E_0, E_1$ are, respectively,  rank $p$ and $q$ vector bundles on $X$, the bosonic reduction of $\calX$. For a coherent sheaf $\calF$ on $\calX$, we define the dual of $\calF$ as the sheaf $\calF^{\vee} = \lhom_{\calO_{\calX}}(\calF,\calO_{\calX}).$

Thanks to \cite[Proposition 2.13]{biblia}, we shall not make any distinction between rank $(p,q)$ locally free sheaves  and rank $(p,q)$ supervector bundles on $\calX$.

The next lemma shows that the operation of taking the reduced part of a sheaf commutes with tensor products, \blue{and with taking duals in the locally free case,} and will be needed in what follows.

\begin{lemma}\label{commutesreduced}
   \blue{ Let $\calF$ and $\calG$ be $\Zn_2$-graded $\calO_\calX$-modules}   on a   ringed superspace $\calX$, and $\calE$ a locally free sheaf on $\calX$. Then the following holds:

    \begin{itemize}
        \item[a)] $\calF_{red} \tens{\calO_X} \calG_{red} =  (\calF \tens{\calO_{\calX}} \calG)_{red}$;\\[-10pt]
        \item[b)] $(\calE^{\vee})_{red} = (\calE_{red})^{\vee}$.
        \end{itemize}
\end{lemma}
\begin{proof}
For part a) note that:
\begin{equation*}
    \calF_{red} \tens{\calO_X} \calG_{red} = \calF \tens{\calO_{\calX}} \calO_X \tens{\calO_X} \calO_X \tens{\calO_{\calX}} \calG = \calF \tens{\calO_{\calX}} \calG \tens{\calO_{\calX}} \calO_X = (\calF \tens{\calO_{\calX}} \calG)_{red}.
\end{equation*}

For part b),  we have 
$$(\calE_{red})^{\vee} = \lhom_{\calO_X}(\calE_{red},\calO_X) = \lhom_{\calO_X}(\calE\tens{\calO_\calX}\calO_X,\calO_X) =   \lhom_{\calO_{\calX}}(\calE, \calO_X)$$
where we have used the Tensor-Hom adjunction. On the other hand we have 
$$ (\calE^\vee)_{red} =  \lhom_{\calO_{\calX}}(\calE, \calO_{\calX}) \tens{\calO_{\calX}} \calO_X $$
so we need to define as isomorphism 
\begin{equation} \label{homred}  \lhom_{\calO_{\calX}}(\calE, \calO_{\calX}) \tens{\calO_{\calX}} \calO_X \to  \lhom_{\calO_{\calX}}(\calE, \calO_X).  \end{equation}
This is defined as
$$\psi\otimes f \mapsto \phi,  \qquad \phi(s) = f \widetilde{\psi(s)}$$
where $s$ is a section of $\calE$ and a tilde denotes the action of the structural morphism $\calO_{\calX} \to \calO_X$.
Since $\calE$ is locally free we only need to show that this is an isomorphism when $\calE=\calO_{\calX}$ and $\calE=\Pi\calO_{\calX}$, which fact is trivial.  Note that the morphism \eqref{homred} is defined even when $\calE$ is not locally free.\end{proof}
A counterexample to part b) when $\calE$ is not locally free will be given in Remark \ref{counter}.

 \smallskip
 \paragraph{\bf Projected and split superschemes.} There is a canonical closed immersion of locally ringed superspaces $i:X \to \calX,$ induced by the epimorphism $\mathcal{O}_{\calX} \to \mathcal{O}_{X}$. We shall say that $\calX$ is projected if there is a morphism of locally ringed superspaces $p:\calX \to X$, such that $p \circ i = Id$, the identity morphism. We shall call $p$ the projection map of $\calX$. The existence of such a map causes the fundamental exact sequence \eqref{superfund} to split, \blue{so that
\begin{equation}\label{Osplit} \calO_\calX = \calO_X \oplus \calJ.
\end{equation}
 as superrings.}

Topologically $p$ is the identity, but the morphism
 $\calO_X\to\calO_\calX$ gives any $\Zn_2$-graded $\calO_\calX$-module a structure of ($\Zn_2$-graded) $\calO_X$-module. To emphasize this, we rewrite equation \eqref{Osplit} as
 \begin{equation}\label{Osplit2} p_\ast \calO_\calX = \calO_X \oplus p_\ast\calJ.
\end{equation}
More generally, if $\calE$ is a locally free  $\Zn_2$-graded $\calO_\calX$-module, by tensoring equation \eqref{Osplit} by $\calE$ be obtain an isomorphism
 \begin{equation}\label{Osplit3} p_\ast \calE = \calE_{red} \oplus p_\ast (\calJ\calE).
\end{equation}

\begin{remark} When $\calX$ is projected, one can define a morphism 
$$   \lhom_{\calO_{\calX}}(\calE, \calO_X) \to \lhom_{\calO_{\calX}}(\calE, \calO_{\calX}) \tens{\calO_{\calX}} \calO_X   $$
which is a right inverse to the morphism \eqref{homred}. Thus when $\calX$ is projected the natural morphism $(\calF^\vee)_{red} \to (\calF_{red})^\vee$ is surjective
for any $\Zn_2$-graded $\calO_\calX$-module $\calF$.
\end{remark}

As noted in \cite[Section 2.8]{biblia}, the cohomology of  a $\mathbb{Z}_2 $-graded $ \mathcal{O}_{\calX}$-module   on a superscheme $ \calX = (X, \mathcal{O}_{\calX}) $ is the same as the cohomology of  $\mathcal{F}$ as a sheaf of abelian groups on $X$. In the case $\calX$ is projected with projection $p\colon\calX\to X$, we have then
\begin{equation}\label{samecohom} H^i(X, \calF) = H^i(X, p_{*}\calF).
\end{equation}
Then  from equation \eqref{Osplit3} we have 
\begin{equation}\label{superreduced}
    H^i(X, \calE) = H^i(X, \calE_{red}) \oplus  H^i(X, p_*(\calJ \calE)).
\end{equation}

\begin{definition}\label{splitdef}
     Let $\calX = (X, \calO_{\calX})$ be a superscheme. 
     We shall say that $\calX$ is split if $\calJ/\calJ^2$ is a locally free finitely generated $\calO_X$-module, and $\calO_\calX = \bigwedge \calJ/\calJ^2$.
    \end{definition} 
Note that in this case $\operatorname{Gr}_i(\mathcal{O}_{\calX})= \bigwedge^i \calE$.
Clearly, split superschemes are projected. When $\calX$ is split equation \eqref{Osplit3} may be strengthened into
$$ p_\ast \calE  = \operatorname{Gr}(\calE)$$
whenever $\calE$ is a locally free  $\Zn_2$-graded $\calO_\calX$-module.

\smallskip
\paragraph{\bf Projective superschemes.}
One defines the projective superspectrum $\mathbb{P}\mathrm{roj}(\bB)$ of a $\Zn$-graded superring $\bB$  by mimicking the construction of the projective spectrum (see \cite[Section 2.3]{biblia} for details). In this work we shall be most interested in the superschemes $\Pn^{n|m}$,  the projective superspectra  of the free polynomial algebra $\mathbb{K}[x_0,\cdots,x_n, \theta_1, \cdots, \theta_m]$, where $\mathbb{K}$ is an algebraically closed field of characteristic 0, $x_i$ are even variables and $\theta_j$ are odd variables. Coherent, torsion-free, reflexive, and locally free sheaves on superschemes are defined in the usual way, and, together with their $\mathbb{Z}_2$-graded morphisms, they form   corresponding categories.

By \cite[Proposition 2.7]{biblia},  $\Pn^{n|m}$ is a split superscheme, with  $\mathcal{O}_{\Pn^{n|m}} = \bigwedge \mathcal{O}_{\mathbb{P}^n}(-1)^{\oplus m}$. Consequently, $\Pn^{n|m}$ is also projected.

For a projective superscheme $\calX = \mathbb{P}\mathrm{roj}(\bB)$ one defines the rank $1|0$ sheaves $\calO_{\calX}(r)$ as the localizations of  the supermodules $\bB[r]$ (See \cite[Definition 2.9 and Proposition 2.10]{biblia}). For a coherent sheaf $\calF$ on $\calX$, we define 

$$\calF(r):=\calF \tens{\calX} \calO_{\calX}(r).$$

If $\calX$ is a projected superscheme then we can use the projection map $p: \calX \to X$ to relate $\calO_{\calX}(r)$ and $\calO_{X}(r)$. 

\begin{lemma}
Let $\bB$ be a $\Zn$-graded superring,  $\calX =  \mathbb{P}\mathrm{roj}(\bB)$ and $\calF$ a coherent sheaf on $\calX$. If $\calX$ is projected the following holds.

\begin{itemize}
    \item[a)] $p^{*}(\mathcal{O}_{X}(r)) = \mathcal{O}_{\calX}(r). $
    \item[b)] $p_*(\mathcal{F}(a)) = p_*(\calF) \tens{\calO_X} \mathcal{O}_{X}(a).$
\end{itemize}

\end{lemma}

\begin{proof}
For part a) let $J \subset \bB$ denote the ideal generated by the nilpotent elements, and $B:= \bB/J$. Since $\bB$ has a natural strucuture of $B-$module, we see that $p^{*}(\mathcal{O}_{X}(r))$ is the localization of the ring in the left-hand side: 
$$(B \tens{B} \bB)[r] = \bB[r].$$
But the localization of the right-hand side is, by definition, $\mathcal{O}_{\calX}(r)$.

For the second part, observe that:

$$p_*(\calE(a)) = p_*(\mathcal{E}\tens{\calO_{\calX}}\calO_{\calX}(a)) = p_*(\calE \tens{\calO_{\calX}} p^*(\mathcal{O}_{X}(a)) = p_*(\calE) \tens{\calO_X} \calO_X(a). $$

\noindent The first isomorphism follows by definition, the second isomorphism is a consequence of the projection formula, and the third isomorphism comes from part a). \end{proof}


\begin{remark} \label{counter} We give a counterexample to part b) of Lemma \ref {commutesreduced}
when $\calE$ is not locally free. Consider the inclusion map $i:\Pn^{1|1} \hookrightarrow \Pn^{1|2}$, and note that $i_{*}(\calO_{\Pn^{1|1}})^\vee =  0$, whence $(i_{*}(\calO_{\Pn^{1|1}})^{\vee})_{red}= 0$. On the other hand, tensoring  the short exact sequence
  $$0 \to \mathcal{I}_{\Pn^{1|1}} \to \calO_{\Pn^{1|2}} \to i_{*}(\calO_{\Pn^{1|1}}) \to 0, $$
  \noindent by $\mathcal{O}_{\Pn^1}$, and using that $ \mathcal{I}_{\Pn^{1|1}}\tens{\calO_{\Pn^{1|2}}} \calO_{\Pn^{1}} = 0$,  we get that:
 $$ i_{*}(\calO_{\Pn^{1|1}})_{red} := i_{*}(\calO_{\Pn^{1|1}}) \tens{\calO_{\Pn^{1|2}}} \calO_{\Pn^{1}}  = \calO_{\Pn^{1}}, $$
 \noindent so that $i_{*}((\calO_{\Pn^{1|1}})_{red})^{\vee} = \calO_{\Pn^{1}}$.
  \end{remark}

With formula   \eqref{Osplit3} and Definition \ref{splitdef} we can recover  the formula for the grading of $\calO_{\Pn^{n|m}}$ on $\Pn^n$ (explicitly given in \cite[Equation (15)]{CN2018}). 
Working in some more generality, consider a split superscheme 
$\calX = (X, \bigwedge \mathcal L^{\oplus m})$
 where $X$ is any scheme, and $\mathcal L$ a line bundle on it. Denote by  $p:\calX \to X$ then natural projection map. Then 
\begin{equation*}
p_*(\mathcal{O}_{\calX}) = \bigoplus_{k = 0}^{\lfloor m/2 \rfloor} (\mathcal L^{2k}) ^{\oplus {m \choose 2k}} \oplus \Pi\bigoplus_{k = 0}^{\lfloor m/2 \rfloor-\frac12[1+(-1)^m]} (\mathcal L^{2k+1})^{\oplus {m \choose 2k+1}}.
 \end{equation*}
Additionally, for a locally free sheaf $\calE$ on $\calX$, we shall have:
\begin{equation}\label{EPn}
p_*(\calE) = \bigoplus_{k = 0}^{\lfloor m/2 \rfloor} (\calE_{red} \otimes \mathcal L^{2k})^{\oplus {m \choose 2k}} \oplus \Pi\bigoplus_{k = 0}^{\lfloor m/2 \rfloor-\frac12[1+(-1)^m]}  (\calE_{red} \otimes \mathcal L^{2k+1})^{\oplus {m \choose 2k+1}}.
 \end{equation}

We now define the super version  of the Rao modules of coherent sheaves on a superscheme. 
 
 \begin{definition} Let $\calX$ be a superscheme and let $\mathcal L$ be an ample line bundle on it \cite{biblia}. Given a coherent sheaf  $\calF$  on $\calX$ let $\calF(a) = \calF\otimes\mathcal L^a$. The Rao module associated with these data is 

     $$\displaystyle \bigoplus_{a \in \mathbb{Z}}H^i(X, \calF(a)).$$
It  has a structure of a $\mathbb{Z}$-graded module, defined by the twists indexed by $a$. Furthermore, since each $H^i(X, \mathcal{F}(a))$ is a $\mathbb{Z}_2$-graded module, it follows that $H^i_*(X, \mathcal{F})$ inherits the structure of a supermodule.
 \end{definition} 
 
 The following statement will summarize the results obtained  so far in this section, aiming to establish a relationship between the Rao module of a coherent sheaf on a projected superscheme and its image under the action of the projection map. The first claim   follows  from the equation \eqref{superreduced}. The second claim follows from equations \eqref{superreduced} and \eqref{EPn}. 

\begin{theorem}\label{reduced=super}
    Let $\calX$ be a projected superscheme and $\calE$ a locally free sheaf on $\calX$. Then $H^{i}_*(X, \calE)$ vanishes, if, and only if, both $H^{i}_*(X, \calE_{\mathrm{red}})$  and   $H^{i}_*(X, p_*(\calJ \calE))$ vanish. Moreover, if $\calX = (X, \bigwedge \mathcal L^{ - \oplus m})$, with $m$  a non-negative integer, and $\mathcal L$ an ample line bundle on $X$,  then $H^{i}_*(X, \calE) = 0$ if and only if $H^{i}_*(X, \calE_{\mathrm{red}}) = 0$.
\end{theorem}

For more  on cohomology of coherent sheaves on superprojective schemes see \cite[Sections 2.7 and 2.8]{biblia}.

We end this section briefly discussing the relationship between a morphism $\varphi \in \lhom_{\calO_{\calX}}(\calF, \calE)$ and its bosonic reduction, that is, its image  $\tilde{\varphi} \in \lhom_{\calO_{X}}(\calF_{red}, \calE_{red})$, \blue{$\tilde\varphi=\varphi\otimes 1$},
for $\calF, \calE$ locally free sheaves of rank $(p_1|q_1)$ and $(p_2|q_2)$, respectively,  on a superscheme $\calX$. 

By taking a simultaneous trivializing open cover of $\calX$, it is enough to make this discussion for homomorphisms $f \in Hom_{\bB}(M,N)$ and its image $\tilde{f} \in Hom_{B}(M/JM,N/JN)$ for free modules $M$ and $N$ of rank $(p_1|q_1)$ and $(p_2|q_2)$, respectively, for a superring $\bB$ with nilpotent ideal $J$. Such $f$ has the following representation:

$$
f = \begin{bmatrix}
A & B \\
C & D
\end{bmatrix},
$$

\noindent where $A,B,C$ and $D$ are matrices with coefficients in $\bB$ and have orders  equals to $p_1\times p_2, p_1\times q_2, q_1\times p_2$ and $p_2 \times q_2$ respectively, with $A$ and $B$ being parity preserving (even) morphisms and $C,D$ being parity reversing (odd) morphisms. For each $a \in \bB$, we denote its image in $\bB/J$ by  $\tilde{a}$. We thus can define for each matrix $Q = (q)_{i,j}$ with entries in $\bB$ the corresponding matrix $\tilde{Q} = (\tilde{q})_{i,j}$, and thus, for $f$ as given above, we define:

$$
\tilde{f}= \begin{bmatrix}
\tilde{A} & \tilde{B} \\
\tilde{C} & \tilde{D}
\end{bmatrix}.
$$

\noindent In particular, since the entries of $B$ and $C$ are odd elements, we have that $\tilde{B} = 0$ and $\tilde{C} = 0$, then, we actually have:

$$
\tilde{f}= \begin{bmatrix}
\tilde{A} & 0 \\
0 & \tilde{D}
\end{bmatrix}.
$$

For $p_1 = p_2 = p$ and $q_1=q_2=q$ we can discuss the invertibility   $f$ based on the invertibilty of $\tilde{f}$.  

Now, by \cite[Corollary 3.1]{BBH2012}, $f$ is invertible  if, and only if, $A$ and $D$ are invertible, and the same applies to  $\tilde{f}$, that is, $\tilde{f}$ is invertible if, and only if, $\tilde{A}$ and $\tilde{D}$ are invertible. So, the question of the inversibility of $f$, provided the inversibility of $\tilde{f}$ boils down to proving the following:

\begin{lemma}\label{inversible}
    Let $\psi : \bB \to \bB$ be a superring homomorphism, $\tilde{\psi}$ its bosonic reduction and $x \in \bB$. Then $\psi(x)$ admtis a multiplicative inverse in $\bB$ if and only if $\tilde{\psi}(x)$ admits a multiplicative inverse in $B = \bB/J$, where $J$ is the ideal of the odd elements.   
\end{lemma}

\begin{proof}
    The ``only if ''part is straightfoward, so that we only prove the converse. Let $\psi : \bB \to \bB$ as in the statement, and suppose that $\tilde{\psi}(x)$ admits a multiplicative inverse in $B = \bB/J$. Since $\tilde{\psi}(x) \in \bB/J$, we can write $\psi(x) = \psi_0(x)+\psi_J(x)$, where $\tilde{\psi}(x) = \psi_0(x)$ and as by hypothesis $\tilde{\psi}(x)$ admits a multiplicative inverse, it follows that $\frac{1}{\psi_0(x)} \in \bB$ and $\psi_J(x) \in J$. Additionally, as $J$ is a nilpotent ideal, it follows that there exists $N \in \mathbb{N}$, such that $\psi_J(x)^N = 0$. Now, we have that 
    $$\psi(x) = \psi_0(x)+\psi_J(x) = \psi_0(x)\left(1+\frac{\psi_J(x)}{\psi_0(x)}\right), $$
and thus
\begin{multline*}
\displaystyle \frac{1}{\psi(x)} = \frac{1}{\psi_0(x)} \cdot \frac{1}{\left( 1 + \frac{\psi_J(x)}{\psi_0(x)} \right)} =  \\ \frac{1}{\psi_0(x)} \left( 1 - \frac{\psi_J(x)}{\psi_0(x)} + \left( \frac{\psi_J(x)}{\psi_0(x)} \right)^2 + \cdots + (-1)^{N-1} \left( \frac{\psi_J(x)}{\psi_0(x)} \right)^{N-1} \right).
\end{multline*}
is the inverse of $\psi(x)$. 
\end{proof}

Now, we apply the Lemma \ref{inversible} to our setting. Let $\psi_1 = \det A$ and $\psi_2 = \det D$. It is easy to see that $\tilde{\psi_1} = \det \tilde{A}$ and $\tilde{\psi_2} = \det \tilde{D}$. Additionally, we have that $\psi_1 $ and $\psi_2$ are invertible if, and only if, $\tilde{\psi_1}$ and $\tilde{\psi_2}$ are invertible. Wich implies that $f$ is invertible if, and only if $\tilde{f}$ is invertible.  So we have proved the following result.

\begin{theorem}\label{bosonicequivalence}
 Let $\calX = (X, \calO_{\calX})$ be a projective superscheme, $\calF, \calE$ locally free sheaves both of rank $p|q$ on $\calX$. Then $\varphi \in \lhom_{\calO_{\calX}}(\calF, \calE)$ is \blue{an isomophism}  if and only if  its bosonic reduction,   $\tilde{\varphi} \in \lhom_{\calO_{X}}(\calF_{red}, \calE_{red})$, $\tilde\varphi = \varphi \otimes 1$, is \blue{an isomophism}.
\end{theorem}
    
\bigskip
\section{Splitting criteria  for supervector bundles on superprojective spaces}

In this section we focus on studying the splitting behavior of supervector bundles on $\Pn^{n|m}$, based on the vanishing of their intermediate Rao modules.  On the usual projective spaces $\Pn^n$ Horrocks' splitting criteria is the following:

\begin{theorem}[Horrocks \cite{H1964}]\label{Horrocks}
    Let $E$ be a rank $r$ vector bundle on $\Pn^n$. Then $H^i_*(\Pn^n, E)$ vanishes for $1 \leq i \leq n-1$ if, and only if, there exists integers $r$ integers $a_1 \geq a_2 \geq \cdots \geq a_r$ such that:

     $$E = \bigoplus_{i=1}^r \calO_{\Pn^n}(a_i).$$ 
\end{theorem}

 The classical proof proceeds by first establishing the result for $\mathbb{P}^1$,  which is  the Birkhoff-Grothendieck splitting criterion, and then applying induction via the restriction sequence associated with a hyperplane to show that the criterion holds for any $\mathbb{P}^n$.

A natural approach  would be to attempt to adapt the classical proof to our setting. However, as  seen in  \cite[Theorem 2]{CN2018}, for $m \geq 2$ the even Picard group of $\mathbb{P}^{1|m}$, which classifies line bundles of rank $1|0$, is not discrete, and  there are rank $1|0$ 
sheaves that are not of the form $\mathcal{O}_{\mathbb{P}^{1|m}}(a)$ for some $a \in \mathbb{Z}$, see  \cite[Example 1]{CN2018}.
For the case of $\mathbb{P}^{1|1}$, the even Picard group is isomorphic to $\mathbb{Z}$, but as our next example will show, there are locally free sheaves  that are not  sums of sheaves of the type  $\mathcal{O}_{\mathbb{P}^{1|1}}(a)$ and $\Pi \mathcal{O}_{\mathbb{P}^{1|1}}(b)$.

\begin{example}
   The supervector bundle $T_{\mathbb{P}^{1|1}}$, the tangent bundle to
  $\mathbb{P}^{1|1}$,
   does not split as sums of line bundles.  Suppose indeed that it does, then
$$T_{\mathbb{P}^{1|1}} = \mathcal{O}_{\Pn^{1|1}}(a) \oplus \Pi \mathcal{O}_{\Pn^{1|1}}(b)$$
for some $a$ and $b$.
 Hence  we would have
$$Hom_{\Pn^{1|1}}(T_{\mathbb{P}^{1|1}}, T_{\mathbb{P}^{1|1}}) = Hom_{\Pn^{1|1}}(\mathcal{O}_{\Pn^{1|1}}(a) \oplus \Pi \mathcal{O}_{\Pn^{1|1}}(b), \mathcal{O}_{\Pn^{1|1}}(a) \oplus \Pi \mathcal{O}_{\Pn^{1|1}}(b)), $$
so that 
$$Hom_{\Pn^{1|1}}(T_{\mathbb{P}^{1|1}}, T_{\mathbb{P}^{1|1}}) = \mathbb{C}^{2|0} \oplus \Pi H^0(\Pn^{1}, \mathcal{O}_{\Pn^{1|1}}(a-b)) \oplus \Pi H^0(\Pn^{1}, \mathcal{O}_{\Pn^{1|1}}(b-a)). $$

So   we have:
\begin{equation}\label{dimHomT}
\dim Hom_{\Pn^{1|1}}(T_{\mathbb{P}^{1|1}}, T_{\mathbb{P}^{1|1}}) =
\begin{cases} 
{2|a-b+1} & \text{if } a-b>0; \\
 {2|b-a+1} & \text{if } a-b<0; \\
 {2|2} & \text{if } a=b. 
\end{cases}
\end{equation}

 On the other hand, consider the Euler exact sequence and its dual (see \cite{CN2018} for details):
$$0 \to \mathcal{O}_{\mathbb{P}^{1|1}} \to \mathbb{C}^{2|1}\otimes \mathcal{O}_{\mathbb{P}^{1|1}}(1) \to T_{\mathbb{P}^{1|1}} \to 0, $$
$$0 \to  \Omega_{\mathbb{P}^{1|1}}\to \mathbb{C}^{2|1}\otimes \mathcal{O}_{\mathbb{P}^{1|1}}(-1) \to \mathcal{O}_{\mathbb{P}^{1|1}} \to 0.  $$
We can compute the cohomologies (in this low dimensional case, it is just a diagram chasing, or we can use  \cite[Corollary 4.10]{GS2024}) to see that:$$H^0(\mathbb{P}^1, \Omega_{\mathbb{P}^{1|1}}(1)) = H^1(\mathbb{P}^1, \Omega_{\mathbb{P}^{1|1}}(1)) = H^0(\mathbb{P}^1 ,\Omega_{\mathbb{P}^{1|1}})=0,$$
\noindent and
$$H^1(\mathbb{P}^1, \Omega_{\mathbb{P}^{1|1}})=\mathbb{C}^{3|1}.$$

So   tensoring the Euler sequence by $\Omega_{\mathbb{P}^{1|1}}$, and considering its long exact sequence of cohomology, we   have  
$$H^0(\mathbb{P}^{1}, T_{\mathbb{P}^{1|1}} \otimes \Omega_{\mathbb{P}^{1|1}} )  =   H^1(\mathbb{P}^{1}, \Omega_{\mathbb{P}^{1|1}}) . $$
Since $Hom_{\Pn^{1|1}}(T_{\mathbb{P}^{1|1}}, T_{\mathbb{P}^{1|1}}) = H^0(\mathbb{P}^{1}, T_{\mathbb{P}^{1|1}} \otimes \Omega_{\mathbb{P}^{1|1}} )$, we have   $\dim Hom_{\Pn^{1|1}}(T_{\mathbb{P}^{1|1}}, T_{\mathbb{P}^{1|1}}) = 3|1$, which contradicts   equation \eqref{dimHomT}. This implies that $T_{\Pn^{1|1}}$ cannot split as a sum  of line bundles. 
 \end{example}
One could argue, as suggested in \cite{Noja2018}, that the correct notion   to generalize invertible sheaves to  the super setting are the so called $\Pi$-invertible sheaves, e.g., rank $1|1$-locally free sheaves $\mathcal{L}$, such that there is an odd endomorphism $\Pi: \mathcal{L} \to \mathcal{L}$, exchanging their even with their odd part; so we should look for a splitting criteria for locally free sheaves in terms of $\Pi$-invertible sheaves. But, by \cite[Corollaries 1 and 2]{CN2018}, the only $\Pi$-invertible sheaves on $\mathbb{P}^{1|1}$ are of the form $\mathcal{O}_{\Pn^{1|1}}(a) \oplus \Pi \mathcal{O}_{\Pn^{1|1}}(a)$, hence we have that $T_{\mathbb{P}^{1|1}}$ is not $\Pi$-invertible, therefore such a criterion  would not hold even in $\mathbb{P}^{1|1}$.
 
We are now ready to state and prove the main result of this work.
\begin{theorem}
Let $\mathcal{E}$ be a rank $p|q$ supervector bundle on $\Pn^{n|m}$, with $n \geq 2$ and $m \geq 1,$ such that $H^{i}_*(\Pn^{n}, \mathcal{E}) = 0$, for $ 1 \leq i \leq n-1$. Then there exist two sequences of integers $a_1 \geq a_2 \geq \ldots, \geq a_p$ and $b_1 \geq b_2 \geq \ldots, \geq b_q$ such that:
 \[\mathcal{E}  =  \displaystyle \bigoplus^{p}_i \mathcal{O}_{\Pn^{n|m}}(a_i) \bigoplus^{q}_i \Pi \mathcal{O}_{\Pn^{n|m}}(b_i) \]    
\end{theorem}

\begin{proof}
    Fix the integers $n \geq 2$ and $m \geq 1$. By   Theorem \ref{reduced=super} we have 
\begin{equation}\label{vanishing0}
    H^i_*(\mathbb{P}^{n}, \mathcal{E}) = H^i_*(\mathbb{P}^{n}, \mathcal{E}_{red}) = 0
\end{equation}
 for every $1 \leq i \leq n-1$. Now, there exists  two vector bundles $E_0, E_1$ on $\Pn^n$, of rank $p$ and $q$ respectively, such that $\calE_{red} = E_0 \oplus \Pi E_1$. By   equation   \eqref{vanishing0} we have 
 $$H^i_*(\mathbb{P}^{n}, E_0) = H^i_*(\mathbb{P}^{n}, E_1) =  0.$$
By  Horrocks' splitting criterion (Theorem \ref{Horrocks}) we have that  $E_0$ and $E_1$ split as sums of  line bundles; this implies that:
\[\mathcal{E}_{red}  = \left[ \displaystyle \bigoplus^{p}_i \mathcal{O}_{\Pn^{n}}(a_i) \right]\oplus\left[\bigoplus^{q}_i \Pi \mathcal{O}_{\Pn^{n}}(b_i)\right]. \]
Now define
$$\mathcal{F}  = \left[  \displaystyle \bigoplus^{p}_i \mathcal{O}_{\Pn^{n|m}}(a_i)\right]\oplus\left[ \bigoplus^{q}_i \Pi \mathcal{O}_{\Pn^{n|m}}(b_i)\right].$$
We are going to prove that there exists an isomorphism $\phi: \mathcal{F} \to \mathcal{E}$. By our construction  there exists an isomorphism $\varphi: \mathcal{F}_{red} \to \mathcal{E}_{red}$, since both can be written as $\oplus^{p}_i \mathcal{O}_{\Pn^{n}}(a_i) \oplus^{q}_i \Pi \mathcal{O}_{\Pn^{n}}(b_i)$. Note also that  
    $$\varphi \in Hom(\mathcal{F}_{red}, \mathcal{E}_{red}) = H^0(\mathbb{P}^n, \mathcal{E}_{red}\otimes \mathcal{F}_{red}^{\vee}) = H^0(\mathbb{P}^n, (\mathcal{E}\tens{\calO_{\Pn^{n|m}}} \mathcal{F}^{\vee})_{red}), $$
where the last isomorphism follows from Lemma \ref{commutesreduced}b).  From the fundamental exact sequence, for $\calE \otimes \calF^{\vee}$ we shall have: 
$$0 \to \calJ(\calE \otimes \calF^{\vee}) \to   \calE \otimes \calF^{\vee} \to (\calE \otimes \calF^{\vee})_{red} \to 0.$$
From the long exact sequence of cohomology, we see that the obstruction to the lifting of the isomorphism $\varphi \in H^0(\mathbb{P}^n, (\mathcal{E} \otimes \mathcal{F}^{\vee})_{red})$ lies in the group $H^1(\Pn^{n}, \calJ(\calE \otimes \calF^{\vee}))$.  By equation
\eqref{superreduced}  we have 
$$H^1(\Pn^{n}, \calJ(\calE \otimes \calF^{\vee})) = H^1(\Pn^{n}, p_*(\calJ(\calE \otimes \calF^{\vee}))).$$
By   Theorem \ref{reduced=super}  it is enough to prove that $H^1(\Pn^{n}, \calE \otimes \calF^{\vee}) = 0$, as $H^1(\Pn^{n}, p_*(\calJ(\calE \otimes \calF^{\vee})))$ is a direct summand of it. But since
    $$\calE \otimes \calF^{\vee} = \left[\bigoplus_{i=1}^p\calE(-a_i)\right]\oplus\left[\bigoplus_{i=1}^q \Pi\calE(-b_i)\right], $$
we have that 
 \begin{equation}\label{tecnico}
      H^1(\Pn^{n}, \calE \otimes \calF^{\vee}) = \left[\bigoplus_{i=1}^p   H^1(\Pn^{n}, \calE(-a_i))\right]\oplus\left[ \bigoplus_{i=1}^q H^1(\Pn^{n}, \calE(-b_i))\right],
  \end{equation}
 and the left-hand side vanishes by hypothesis. It then follows that the right-hand side of the equation \eqref{tecnico} vanishes, which implies that $H^1(\Pn^{n|m}, \calJ(\calE \otimes \calF^{\vee})) = 0$, and thus the isomorphism $\varphi \in Hom(\mathcal{F}_{red}, \mathcal{E}_{red})$ lifts to a morphism $\phi \in \Hom_{\calO_{\Pn^{n|m}}}(\mathcal{F}, \mathcal{E})$. 
  
Since $\tilde{\phi} = \varphi$, the result   follows by the Theorem \ref{bosonicequivalence}.
\end{proof}

We draw the reader's attention to the fact that equation \eqref{superreduced} and Theorem \ref{reduced=super} are crucial to the proof of this result, and for their proofs we relied heavily on the split structure of the superscheme. Thus generalizing this result for other superschemes such as supergrassmannians as in \cite{O89}, would require further adaptations of the induction process.

\end{document}